\documentclass[10pt]{article}
 \makeatletter
\newfont{\footsc}{cmcsc10 at 8truept}
\newfont{\footbf}{cmbx10 at 8truept}
\newfont{\footrm}{cmr10 at 10truept}
\renewcommand{\ps@plain}{%
\renewcommand{\@oddfoot}{\footsc {\footbf }  \footrm\thepage}}
\makeatother  \leftmargin=25mm
\usepackage[centertags]{amsmath}
\usepackage{amsfonts}
\usepackage{amsthm}
\usepackage{newlfont}
\usepackage{graphicx}
\usepackage{epstopdf}
\usepackage{subfigure}
\usepackage[utf8]{inputenc}
\usepackage{indentfirst}
\usepackage{latexsym}
\usepackage{float}
\usepackage{cite}
\usepackage{CJK}

\usepackage{todonotes}

\newtheorem{thm}{Theorem}[section]

\newtheorem{lem}[thm]{Lemma}

\theoremstyle{definition}

\theoremstyle{remark}
\newtheorem{rem}[thm]{Remark}

\DeclareMathOperator{\conv}{conv}
\DeclareMathOperator{\aff}{aff}
\DeclareMathOperator{\sgn}{sgn}

\title{On lattice coverings by locally anti-blocking bodies and polytopes with few vertices}
\date{}
\author{Matthias Schymura\footnote{
Institute of Mathematics, University of Rostock, Germany, E-mail:  matthias.schymura@uni-rostock.de.}, Jun Wang\footnote{Corresponding author. School of Mathematical Sciences \& Institute of Mathematics and Interdisciplinary Sciences, Tianjin Normal University, Tianjin, China, 300387. E-mail: junwang@tjnu.edu.cn. }, and Fei Xue\footnote{School of Mathematical Sciences, Nanjing Normal University, Nanjing, China, 210046. E-mail: 05429@njnu.edu.cn.} }

\begin{document}
\maketitle
\begin{center}\vskip -0.8cm
{\small 
 \vskip 0.3cm
}
\end{center}

\begin{abstract}In 2021, Ordentlich, Regev and Weiss made a breakthrough that the lattice covering density of any $n$-dimensional convex body is upper bounded by $cn^{2}$, improving on the best previous bound established by Rogers in 1959. However, for the Euclidean ball, Rogers obtained the better upper bound $n(\log_{e}n)^{c}$, and this result was extended to certain symmetric convex bodies by Gritzmann. The constant $c$ above is independent on $n$. In this paper, we show that such a bound can be achieved for more general classes of convex bodies without symmetry, including anti-blocking bodies, locally anti-blocking bodies and $n$-dimensional polytopes with $n+2$ vertices.
\\


\end{abstract}

\maketitle
\section{Introduction}

Let $K$ be an $n$-dimensional convex body in $\mathbb{R}^{n}$, i.e.~a compact and convex set with non-empty interior. Denote by $\mathcal{K}^{n}$ the set of convex bodies in $\mathbb{R}^{n}$. For a measurable subset $S \subseteq \mathbb{R}^n$, we denote by $V(S)$ its volume, that is, its Lebesgue measure. For a discrete point set~$X$, we call $\Re = K + X := \{K+\mathbf{x}:\mathbf{x}\in X\}$ a covering of space by translates of~$K$, if $K+X=\mathbb{R}^{n}$.
Writing $W_{n}=\{\mathbf{x} \in \mathbb{R}^n : 0\leq|x_{i}|\leq \frac{1}{2}, i=1,\ldots,n\}$ for the $n$-dimensional unit cube, we define the density of a covering $\Re$ as $$\theta(\Re)=\liminf\limits_{l\to\infty}\frac{V(\Re\cap lW_{n})}{V(lW_{n})}.$$
Then, the translative covering density of $K$ is defined as $\theta_{T}(K)=\inf\limits_{\Re} \theta(\Re)$, where the infimum ranges over all coverings $\Re$ by translates of~$K$.
In addition, if we restrict the translation vectors to form a lattice~$\Lambda$ and if we let $\mathcal{L}_K$ denote the family of all lattices $\Lambda$ such that $K+\Lambda$ is a covering of $\mathbb{R}^{n}$, then the lattice covering density of~$K$ is defined by $$\theta_{L}(K)=\inf\limits_{\Lambda \in \mathcal{L}_K}\frac{V(K)}{\det(\Lambda)},$$
where $\det(\Lambda)$ is the determinant of $\Lambda$.
The covering density is a measure of the efficiency of a given covering. Clearly, we have $$1\leq \theta_{T}(K)\leq \theta_{L}(K)<\infty.$$
For more information on covering densities, we refer to~\cite{Fej23,Goo17}.

If we want to prove an upper estimate on $\theta_{L}(K)$, one possible approach is to find a large inscribed body $M\subseteq K$ and turn to estimate $\theta_{L}(M)$.
Then, an upper bound on $\theta_L(K)$ follows because $\mathcal{L}_M \subseteq \mathcal{L}_K$ and thus
\begin{equation}\label{eq01}
\theta_L(K) \leq \frac{V(K)}{V(M)} \theta_L(M) \quad \textrm{ for all } \quad M \subseteq K.
\end{equation}

Let $B^{n}_{p}=\{\mathbf{x}:(|x_{1}|^{p}+|x_{2}|^{p}+\cdots+|x_{n}|^{p})^{\frac{1}{p}}\leq1\}$ be the $n$-dimensional unit $\ell_{p}$-ball and let $S^{n}$ be an $n$-dimensional regular simplex with unit edge lengths. For simplicity, the $n$-dimensional unit Euclidean ball is denoted by $B^{n}$. In lower dimensions, for $2\leq n\leq 5$, the exact value of $\theta_{L}(B^{n})$ has been determined (see~\cite{Sch06} for references). In particular, R.~Kershner~\cite{Ker39} proved that $\theta_{T}(B^{2})=\theta_{L}(B^{2})$. More generally, for all centrally symmetric convex bodies $C$ in $\mathbb{R}^{2}$, L.~Fejes T\'{o}th~\cite{Fej46,Fej50} showed that $\theta_{T}(C)=\theta_{L}(C)$. For non-symmetric ones, J.~Januszewski~\cite{Jan10} in 2010 proved that $\theta_{T}(S^{2})=\theta_{L}(S^{2})$. Later in 2015, K.~Sriamorn and F.~Xue~\cite{Sri15} generalized this result to a class of convex disks (quarter-convex disks, actually $2$-dimensional anti-blocking bodies), which includes triangles and convex quadrilaterals. Up to now, we still do not know if $\theta_{T}(K)=\theta_{L}(K)$ for all $K\in\mathcal{K}^{2}$. In 1990s, R.~Dougherty and V.~Faber~\cite{Dou04}, C.~M.~Fiduccia, R.~W.~Forcade and J.~S.~Zito~\cite{Fid98}, R.~W.~Forcade and J.~Lamoreaux~\cite{For00} considered the degree-diameter problem for Cayley graphs of Abelian groups for both directed graphs and undirected graphs, and gave upper bounds on $\theta_{L}(S^{3})$ and $\theta_{L}(B^{3}_{1})$. Lower bounds for $\theta_{L}(S^{3})$, $\theta_{L}(S^{4})$ and $\theta_{L}(S^{n})$ were studied by M.~Fu, F.~Xue, C.~Zong in recent years, see~\cite{Xue18,Fu23} for references.

In higher dimensions, based on periodic sets and a mean value argument, C. ~A.~Rogers~\cite{Rog57} in 1957 gave an upper bound on the translative covering density $\theta_{T}(K)$ for every $K\in \mathcal{K}^{n}$,
\begin{equation*}    
\theta_{T}(K)\leq n\log_{e}n+n\log_{e}\log_{e}n+5n.
\end{equation*}
Later in 1959, Rogers~\cite{Rog59} proved that for all convex bodies $K\in \mathcal{K}^{n}$, $$\theta_{L}(K)\leq n^{\log_{2}\log_{e}n+c}.$$
In 2021, O.~Ordentlich, O.~Regev and B.~Weiss~\cite{Ord21} improved Rogers' upper bound to
\begin{equation*}
\theta_{L}(K)\leq cn^{2}.
\end{equation*}
In both of these bounds and the bounds below, $c$ is some absolute constant whose actual value may differ in various places.

For the case of $B^{n}$, by estimating the lattice covering densities of $k$-cylinders and approximating convex bodies with inscribed $k$-cylinders, C.~A.~Rogers~\cite{Rog59} in 1959 obtained a better estimate,  \begin{equation}\label{eq03}
\theta_{L}(B^{n})\leq cn(\log_{e}n)^{\frac{1}{2}\log_{2}(2\pi e)}.
\end{equation}
As for the lower bound, H.~S.~M.~Coxeter, L.~Few and C.~A.~Rogers~\cite{Cox59} in 1959 proved that $$\frac{cn}{e\sqrt{e}} \leq \theta_{L}(B^{n}).$$
In 1985, P.~Gritzmann~\cite{Gri85} proved a bound similar to~\eqref{eq03} for a larger class of convex bodies.
He introduced the generalized cylinders to improve the approximation theorem for convex bodies with inscribed  $k$-fold cylinders in Rogers' method, and proved that 
\begin{equation}\label{eq04}
\theta_{L}(K)\leq cn(\log_{e}n)^{1+\log_{2}e},
\end{equation}
for every $K$ that has a non-singular affine image that is symmetric about at least $\log_{2}\log_{e}n+4$ coordinate hyperplanes.

In particular, the above inequality~\eqref{eq04} also holds for $B^{n}_{p}$ and $S^{n}$ (note that one can find  $\lceil n/2\rceil$ mutually orthogonal hyperplanes of symmetry in a suitable affine image of~$S^{n}$). In fact, by Gritzmann's result, we need not require very much symmetry. For example, in dimension $n=10^{9}$, one merely needs~$9$ hyperplanes of symmetry. This motivates us to consider the lattice covering densities for more general classes of convex bodies without symmetry assumptions.

In this paper, we give better estimates on lattice covering densities for all $n$-dimensional anti-blocking bodies, locally anti-blocking bodies and polytopes with $n+2$ vertices. The definition of (locally) anti-blocking bodies is given in Section~\ref{sec3}. Without using the symmetry hyperplanes of convex bodies, we construct cylinders inscribed in such convex bodies from another perspective. In Section~\ref{sec2}, we present two useful lemmas on an approximation result for convex bodies and an estimate of the lattice covering density for $k$-fold cylinders, which are crucial to our proof. In Section~\ref{sec3}, we give a brief introduction to anti-blocking bodies and locally anti-blocking bodies, and estimate their lattice covering densities. In Section~\ref{sec4}, we estimate the lattice covering density for any $n$-dimensional polytope with $n+2$ vertices.

The main theorems are stated as follows.
\begin{thm}\label{thm1}
 Let $F_{n}$ be an $n$-dimensional anti-blocking body in $\mathbb{R}^{n}$, then there is a constant $c$ such that 
 $$\theta_{L}(F_{n})\leq cn(\log_{e}n)^{1+\log_{2}e}.$$
\end{thm}

\begin{thm}\label{thm2}
 Let $L_{n}$ be an $n$-dimensional locally anti-blocking body, then there is a constant $c$ such that
 $$\theta_{L}(L_{n})\leq cn(\log_{e}n)^{2+\log_{2}e}.$$
\end{thm}

\begin{thm}\label{thm3}
 Let $P$ be an $n$-dimensional polytope in $\mathbb{R}^{n}$ with at most $n+2$ vertices,
 then there is a constant $c$ such that
 $$\theta_{L}(P)\leq cn(\log_{e}n)^{1+\log_{2}e}.$$
\end{thm}

\begin{rem}
In view of \eqref{eq04} and Theorems~\ref{thm1}-\ref{thm3}, we may conjecture that for all $K\in\mathcal{K}^{n}$, we have $\theta_{L}(K)\leq cn(\log_{e}n)^{1+\log_{2}e}$, for some absolute constant $c > 0$.
\end{rem}

\section{Some useful lemmas}\label{sec2}

A $k$-fold cylinder $C_{k}$ is defined as the Cartesian product $T\times S_{1}\times\ldots\times S_{k}$ of an $(n-k)$-dimensional convex body $T$ and $k$ line segments $S_{1},\ldots,S_{k}$. In particular, when $k=1$, one gets an ordinary cylinder.
%For a measurable subset $S \subseteq \mathbb{R}^n$, we denote by $V(S)$ its volume, that is, its Lebesgue measure.
A.~M.~Macbeath~\cite{Mac51} proved that for every $K\in \mathcal{K}^{n}$, there exists a cylinder $C_{1}$ inscribed in $K$ with $$V(K)\leq n\left(1+\frac{1}{n-1}\right)^{n-1}V(C_{1}).$$ Based on that, C.~A.~Rogers~\cite{Rog59} further gave an approximation theorem for convex bodies with inscribed $k$-fold cylinders and a rather efficient covering of $C_{k}$.

\begin{lem}[Rogers \cite{Rog59}]\label{lem1}
 Let $n$ be sufficiently large, let $k$ be an integer satisfying $n>k>\log_{2}\log_{e}n+4$, and let $C_{k}$ be a $k$-fold cylinder. Then
 $$\theta_{L}(C_{k})\leq 2^{k}\left[\frac{1}{4}(n-k)\log_{e}\frac{27}{16}-3\log_{e}(n-k)\right]\left(1+\frac{1}{n}\right)^{n}.$$
\end{lem}

P.~Gritzmann~\cite{Gri85} introduced the concept of generalized cylinders to improve Macbeath's result. A subset $Z$ of $\mathbb{R}^{n}$ is called a generalized cylinder if it satisfies the following two conditions: First, there exists a hyperplane~$H$ with normal vector~$\mathbf{u}$ such that the orthogonal projection of~$Z$ onto~$H$ is an $(n-1)$-dimensional convex body. And second, all non-empty intersections of~$Z$ with lines perpendicular to~$H$ are segments of the same non-zero length. The direction of the generalized cylinder is $\mathbf{u}$. It can be seen that such generalized cylinders need not be convex. With this notion, Gritzmann's improved approximation theorem for convex bodies with inscribed generalized cylinders reads as follows:

\begin{lem}[Gritzmann \cite{Gri85}]\label{lem2}
 Let $K\in \mathcal{K}^{n}$ and let $\mathbf{u}\in \mathbb{R}^{n}$. Then there is a generalized cylinder $Z$ of direction $\mathbf{u}$ inscribed in $K$ such that
 $$V(K)\leq \left(1+\frac{1}{n-1}\right)^{n-1}V(Z).$$
\end{lem}

Note that the result allows to prescribe the direction $\mathbf{u}$ of the large generalized cylinder inside~$K$.

\section{Anti-blocking and locally anti-blocking bodies}\label{sec3}

To begin with, we give the definition of anti-blocking and locally anti-blocking bodies.
A convex body $F_{n}\subseteq \mathbb{R}^{n}_{\geq 0}$ is called anti-blocking if for any $\mathbf{x}=(x_{1},\ldots,x_{n})\in F_{n}$, the whole axes-parallel box $[0,x_{1}]\times[0,x_{2}]\times\ldots\times[0,x_{n}]$ is contained in $F_{n}$. Anti-blocking bodies were firstly named and studied by D.~R.~Fulkerson~\cite{Ful71,Ful72} in the context of combinatorial optimization in 1970s, also called convex corners or down-closed bodies~\cite{Bel00}.

For $\sigma\in \{-1,1\}^{n}$ and any set $S \subseteq \mathbb{R}^{n}$, let $\sigma(S)=\{(\sigma_{1}x_{1},\ldots,\sigma_{n} x_{n}) : \mathbf{x}=(x_{1},\ldots,x_{n})\in S\}$. We call an $n$-dimensional convex body $L_{n}$ locally anti-blocking body if $\sigma(L_n)\cap\mathbb{R}^{n}_{\geq0}$ is an anti-blocking body for every $\sigma\in \{-1,1\}^{n}$. In addition, if $L_{n}$ is a polytope, then $L_{n}$ is said to be a locally anti-blocking polytope. Locally anti-blocking polytopes were introduced in~\cite{Koh20} and studied in depth in~\cite{Art23}. Especially, an $n$-dimensional convex body $U_{n}$ is called $1$-unconditional or simply unconditional if for any $\mathbf{x}\in U_{n}$, we have $\sigma(\mathbf{x}) = (\sigma_{1}x_{1},\ldots,\sigma_{n} x_{n}) \in U_{n}$, for every $\sigma\in \{-1,1\}^{n}$. We can see that $U_{n}\cap\mathbb{R}^{n}_{\geq0}$ is an anti-blocking body. Clearly, all unconditional bodies and anti-blocking bodies are locally anti-blocking bodies. 

Locally anti-blocking bodies play an essential role in studying some famous conjectures in Convex Geometry, such as Kalai's $3^{d}$-conjecture, Godbersen's conjecture, Mahler's conjecture and Hadwiger-Gohberg-Markus' illumination conjecture, see~\cite{San24, Sad25, Art23, Tik17}.

\begin{proof}[\textbf{Proof of Theorem~\ref{thm1}}]
Let $\mathbf{e}_{i}=(0,\ldots,0,1,0,\ldots,0)$ be the $i$-th coordinate unit vector in $ \mathbb{R}^{n}$. By Lemma~\ref{lem2}, for any given $n$-dimensional anti-blocking body $F_{n}$, there is a generalized cylinder $Z_{n}$ of direction $\mathbf{e}_{n}$ inscribed in $F_{n}$ such that $V(F_{n})\leq \left(1+1/(n-1)\right)^{n-1}V(Z_{n}).$ In fact, $Z_{n}$ can be assumed to be an ordinary cylinder: First, since $Z_{n}\subseteq F_{n}$ and $F_n$ is anti-blocking, the orthogonal projection~$P_{n-1}$ of~$Z_{n}$ onto $\mathbf{e}_{n}^{\perp} = \left\{ \mathbf{x} \in \mathbb{R}^n : x_n = 0 \right\}$ is an $(n-1)$-dimensional convex body in $F_{n}$. And secondly, suppose that all non-empty intersections of $Z_{n}$ with lines parallel to $\mathbf{e}_{n}$ are segments of the same non-zero length~$a_{1}$. Then, the ordinary cylinder $P_{n-1}\times [0,a_{1}]$ is inscribed in $F_{n}$ and has the same volume as $Z_{n}$. For simplicity, we still use the notation $Z_{n}=P_{n-1}\times [0,a_{1}]=C_{1}$ to denote this ordinary cylinder, and we get $$V(F_{n})\leq \left(1+\frac{1}{n-1}\right)^{n-1}V(C_{1}).$$
Moreover, the set
$$F_{n-1} = P_{n-1} \ \cup \bigcup_{\mathbf{y}=(y_{1},\ldots,y_{n})\in P_{n-1}} [0,y_{1}]\times[0,y_{2}]\times\ldots\times[0,y_{n-1}]$$
is an $(n-1)$-dimensional anti-blocking body in $\mathbf{e}_n^\perp$ satisfying $F_{n-1} \times [0,a_1] \subseteq F_n$.

This process can be repeated inductively, and for $k=1,\ldots,n$, we get an ordinary cylinder $Z_{n-k+1}$ of direction $\mathbf{e}_{n-k+1}$ and height $a_{k}$ in $F_{n-k+1}$ satisfying the volume bound in Lemma~\ref{lem2}.
Further, let $P_{n-k}$ be the orthogonal projection of $Z_{n-k+1}$ along $\mathbf{e}_{n-k+1}$ onto the $(n-k)$-dimensional subspace $\mathbf{e}_{n-k+1}^{\perp}\cap\ldots\cap\mathbf{e}_{n}^{\perp}$.
Then, we have
$$Z_{n-k+1}=P_{n-k}\times[0,a_{k}].$$

The set
\begin{equation}\label{eq06}
%\[
F_{n-k} = P_{n-k} \ \cup \bigcup_{\mathbf{y}=(y_{1},\ldots,y_{n})\in P_{n-k}} [0,y_{1}]\times[0,y_{2}]\times\ldots\times[0,y_{n-k}].
%\]
\end{equation}
is an $(n-k)$-dimensional anti-blocking body satisfying
\begin{equation}\label{eq05}
F_{n-k}\times[0,a_{1}]\times\ldots\times[0,a_{k}]\subseteq F_{n}.
\end{equation}
Thus, we obtain a $k$-fold cylinder $$C_{k}=P_{n-k}\times[0,a_{1}]\times\ldots\times[0,a_{k}],$$ which by~\eqref{eq06} and~\eqref{eq05} satisfies $C_{k}\subseteq F_{n}$.
By induction, we get the volume estimate $$V(F_{n}) \leq \prod_{i=n-k}^{n-1}\left (1+\frac{1}{i}\right)^{i}V(C_{k}),$$
for all $k \geq 1$.
In view of~\eqref{eq01}, the lattice covering densities $\theta_{L}(F_{n})$ and $\theta_{L}(C_{k})$ satisfy
$$\theta_{L}(F_{n})\leq \frac{V(F_{n})}{V(C_{k})}\theta_{L}(C_{k}).$$
By Lemma~\ref{lem1}, let $k$ be the least integer greater than $\log_{2}\log_{e}n+4$, then for sufficiently large $n$, we have
\begin{align*}\label{eq0.3}
\theta_{L}(F_{n})
&\leq 2^{k}\left[\frac{1}{4}(n-k)\log_{e}\frac{27}{16}-3\log_{e}(n-k)\right]\prod_{i=n-k}^{n}\left (1+\frac{1}{i}\right)^{i}\\
&<2^{k}e^{k+1}n\\
&\leq cn(\log_{e}n)^{1+\log_{2}e}.
\end{align*}
The constant $c$ is independent of $n$ and $F_{n}$.
\end{proof}

\begin{proof}[\textbf{Proof of Theorem~\ref{thm2}}]
The proof for locally anti-blocking bodies is quite similar to the one for anti-blocking bodies in Theorem~\ref{thm1} above.
We repeat the arguments with the necessary adjustments that explain the additional $\log_e(n)$ factor in the resulting upper bound on $\theta_L(L_n)$.

Let $H^{+}_{n}=\{\mathbf{x} \in \mathbb{R}^n : x_{n}\geq 0\}$ and $H^{-}_{n}=\{\mathbf{x} \in \mathbb{R}^n : x_{n}\leq 0\}$.
By Lemma~\ref{lem2}, for any given $n$-dimensional locally anti-blocking body $L_{n}$, there is a generalized cylinder~$Z_{n}^{\pm}$ %(or $Z_{n}^{-}$) 
of direction $\mathbf{e}_{n}$ inscribed in $L_{n}\cap H^{\pm}_{n}$ 
such that $V(L_{n}\cap H^{\pm}_{n})\leq \left(1+1/(n-1)\right)^{n-1}V(Z_{n}^{\pm})$. Let $P_{n-1}^{\pm}$ be the orthogonal projection of $Z_{n}^{\pm}$ onto $\mathbf{e}_{n}^{\perp}$ along $\mathbf{e}_{n}$ respectively, then $P_{n-1}^{\pm}$ is an $(n-1)$-dimensional convex body in $L_{n}\cap H^{\pm}_{n}$. 
Just as in the proof of Theorem~\ref{thm1}, $Z_{n}^{\pm}$ can be assumed to be ordinary cylinders. That is to say, there exists $Z_{n}^{\pm}=P_{n-1}^{\pm}\times [0,a_{1}^{\pm}]$ of direction $\mathbf{e}_{n}$ inscribed in $L_{n}\cap H^{\pm}_{n}$ such that $$V(L_{n}\cap H^{\pm}_{n})\leq \left(1+\frac{1}{n-1}\right)^{n-1}V(Z_{n}^{\pm}).$$
Without loss of generality, assume that $V(Z_{n}^{-})\leq V(Z_{n}^{+})$ and let $C_{1}=Z_{n}^{+}=P_{n-1}^{+}\times [0,a_{1}^{+}]$, then we have $$V(L_{n})\leq 2\left(1+\frac{1}{n-1}\right)^{n-1}V(C_{1}).$$
Moreover, for a vector $\mathbf{y} = (y_1,\ldots,y_n)$ we write
\[
B(y_1,\ldots,y_{n-k}) = \left\{ \left(\sgn(y_1) \cdot x_1,\ldots,\sgn(y_{n-k}) \cdot x_{n-k} \right) : 0 \leq x_i \leq |y_i|, i=1,\ldots,n-k\right\}
\]
for the $(n-k)$-dimensional axis-parallel box with diagonally opposite vertices $\mathbf{0}$ and $(y_1,\ldots,y_{n-k})$, where $\sgn(r)$ denotes the sign of a real number~$r$.
With this notation, the set
\[
L_{n-1}^+ = P_{n-1}^+ \ \cup \bigcup_{\mathbf{y}=(y_{1},\ldots,y_{n})\in P_{n-1}^+} B(y_1,\ldots,y_{n-1})\times \{0\}
\]
is an $(n-1)$-dimensional locally anti-blocking body in $\mathbf{e}_n^\perp$ satisfying $L_{n-1}^+ \times [0,a_1^+] \subseteq L_n$.

Again, we repeat this process inductively.
For $k=1,\ldots,n$, we define the $(n-k+1)$-dimensional halfplanes
\begin{align*}
H_{n-k+1}^+ &= \left\{ \mathbf{x} \in \mathbb{R}^n : x_n = \ldots = x_{n-k+2} = 0, x_{n-k+1} \geq 0 \right\} \textrm{ and }\\
H_{n-k+1}^- &= \left\{ \mathbf{x} \in \mathbb{R}^n : x_n = \ldots = x_{n-k+2} = 0, x_{n-k+1} \leq 0 \right\}.
\end{align*}
We find ordinary cylinders $Z_{n-k+1}^{\pm} \subseteq L_{n-k+1} \cap H_{n-k+1}^{\pm}$ of direction $\mathbf{e}_{n-k+1}$ and height $a_{k}^{\pm}$ such that
\[
V(L_{n-k+1} \cap H^{\pm}_{n-k+1})\leq \left(1+\frac{1}{n-k}\right)^{n-k}V(Z_{n-k+1}^{\pm}).
\]
Without loss of generality we may assume $V(Z_{n-k+1}^-) \leq V(Z_{n-k+1}^+)$, so that
\[
V(L_{n-k+1}) \leq 2 \left(1+\frac{1}{n-k}\right)^{n-k}V(Z_{n-k+1}^+).
\]
Let $P_{n-k}^+$ be the orthogonal projection of $Z_{n-k+1}^+$ onto $\mathbf{e}_{n-k+1}^{\perp}\cap\ldots\cap\mathbf{e}_{n}^{\perp}$ along $\mathbf{e}_{n-k+1}$ so that

$$Z_{n-k+1}^+ = P_{n-k}^+ \times[0,a_{k}^+].$$

Further, let
\begin{equation}\label{eq07}
L_{n-k}^+ = P_{n-k}^+\ \cup \bigcup_{\mathbf{y}=(y_{1},\ldots,y_{n})\in P_{n-k}^+} B(y_1,\ldots,y_{n-k})\times \{0\}^k.
\end{equation}
Then $L_{n-k}^+$ is an $(n-k)$-dimensional locally anti-blocking body satisfying
\begin{equation}\label{eq08}
L_{n-k}^+ \times [0,a_{1}^{+}]\times \ldots \times[0,a_{k}^+] \subseteq L_{n}.
\end{equation}
We have thus constructed a $k$-fold cylinder
\[
C_{k} = P_{n-k}^+\times[0,a_{1}^{+}]\times\ldots\times[0,a_{k}^{+}].
\]
By (\ref{eq07}) and (\ref{eq08}), we have $C_{k}\subseteq L_{n}$, and inductively the volume bounds provide us with the estimate
\[
V(L_{n}) \leq 2^{k}\prod_{i=n-k}^{n-1}\left (1+\frac{1}{i}\right)^{i}V(C_{k}),
\]
for every $k \geq 1$.
In view of~\eqref{eq01}, the lattice covering densities $\theta_{L}(L_{n})$ and $\theta_{L}(C_{k})$ satisfy
$$\theta_{L}(L_{n})\leq \frac{V(L_{n})}{V(C_{k})}\theta_{L}(C_{k}).$$
By Lemma~\ref{lem1}, let $k$ be the least integer greater than $\log_{2}\log_{e}n+4$, then for sufficiently large $n$, we have
\begin{align*}\label{eq0.3}
\theta_{L}(L_{n})
&\leq 2^{2k}\left[\frac{1}{4}(n-k)\log_{e}\frac{27}{16}-3\log_{e}(n-k)\right]\prod_{i=n-k}^{n}\left (1+\frac{1}{i}\right)^{i}\\
&<2^{2k}e^{k+1}n\\
&\leq cn(\log_{e}n)^{2+\log_{2}e}.
\end{align*}
The constant $c$ is again independent of $n$ and $L_{n}$.
\end{proof}

\section{$n$-dimensional polytopes with $n+2$ vertices}\label{sec4}

Clearly, every simplex is affinely equivalent to an anti-blocking body of the same dimension, so that Theorem~\ref{thm1} applies.
However, not every $n$-dimensional polytope with $n+2$ vertices has this property (see Remark~\ref{rem1} below).
In this part, we show how Rogers' method can be used nevertheless on this class of polytopes and in particular, we give the proof of Theorem~\ref{thm3}.

Hence, let $P$ be an $n$-dimensional polytope with $n+2$ vertices.
Gr\"unbaum~\cite[Section~6.1]{Gru03} showed that combinatorially $P$ can be described as a $k$-fold pyramid over the free-sum of two lower-dimensional simplices, and that there are $\lfloor n^2/4 \rfloor$ different combinatorial types of such polytopes in $\mathbb{R}^n$.
For our purposes we only need a less detailed description of $P$, which rests on the following concept:
We say that a polytope $Q \subseteq \mathbb{R}^n$ is a generalized bipyramid, if there is an $(n-1)$-dimensional polytope~$B$ in~$\mathbb{R}^n$ (called the base of $Q$) and two points $\mathbf{v},\mathbf{w}$, one on each side of the hyperplane $\aff(B)$, and such that $Q = \conv\{[\mathbf{v},\mathbf{w}] \cup B\}$.
Note that for an ordinary bipyramid one would want the line segment $[\mathbf{v},\mathbf{w}]$ to intersect the relative interior of $B$, here we only need it to intersect the hyperplane $\aff(B)$.

\begin{lem}\label{lem6}
If $P$ is an $n$-dimensional polytope with $n+2$ vertices, then $P$ is a pyramid or a generalized bipyramid.
\end{lem}

\begin{proof}
If $P$ has $n+2$ vertices and is not a pyramid, then one may take a vertex $\mathbf{v}$ of $P$ and look at the simplex $S$ spanned by the remaining $n+1$ vertices. There must be a facet $F$ of $S$, such that $\mathbf{v}$ is on one side of $\aff(F)$, while the vertex of $S$ that is not contained in $F$ is on the other side of $\aff(F)$. If that would not be the case, then $\mathbf{v}$ would be contained in $S$, a contradiction.
Then, calling~$\mathbf{w}$ that vertex on the other side of such a facet~$F$ of~$S$, we obtain the line segment $[\mathbf{v},\mathbf{w}]$ that identifies $P$ to be a generalized bipyramid with base $F$.
\end{proof}

By modifying Rogers' proof of \cite[Lemma~1a]{Rog59} we obtain the following result.

\begin{lem}\label{lem5}
Let $P$ be an $n$-dimensional polytope with $n+2$ vertices, there is a cylinder $C$ inscribed in~$P$ such that
 $$V(P)=\left(1+\frac{1}{n-1}\right)^{n-1}V(C).$$
Moreover, the base of~$C$ is an $(n-1)$-dimensional polytope with $n$ or $n+1$ vertices. 
\end{lem}
\begin{proof}
Let $P$ be an $n$-dimensional polytope with $n+2$ vertices $\{\mathbf{v}_{0},\mathbf{v}_{1},\ldots,\mathbf{v}_{n},\mathbf{v}_{n+1}\}$. By Lemma~\ref{lem6}, if $P$ is a simplicial polytope, then $P$ is a generalized bipyramid. Without loss of generality, let $H=\aff\{\mathbf{v}_{2},\ldots,\mathbf{v}_{n},\mathbf{v}_{n+1}\}$ be the hyperplane containing the base of~$P$, and assume that $\mathbf{v}_{0}$ and $\mathbf{v}_{1}$ are strictly separated by $H$, then $P=\conv\{P'\cup [\mathbf{v}_{0},\mathbf{v}_{1}]\}$, where $P'=\conv\{[\mathbf{v}_{0},\mathbf{v}_{1}] \cap H,\mathbf{v}_{2},\ldots,\mathbf{v}_{n},\mathbf{v}_{n+1} \}$. Clearly, $P'$ is an $(n-1)$-dimensional polytope with $n$ or $n+1$ vertices. After applying a suitable affine transformation, we may assume that $\mathbf{v}_{0}=(0,\ldots,0,0)$, $\mathbf{v}_{1}=(0,\ldots,0,1)$ and $P'\subseteq \{\mathbf{x} \in \mathbb{R}^n : x_{n}=\xi\}$ for some $0 \leq \xi \leq 1$. Let $$C=\conv\left\{\left((1-\frac{1}{n})P'+\frac{1}{n}\mathbf{v}_{1}\right)\bigcup\left((1-\frac{1}{n})P'\right)\right\} = \left(1-\frac{1}{n}\right)P' + [\mathbf{0},\frac{1}{n}\mathbf{v}_1].$$
By definition, $C$ is a cylinder inscribed in $P$ with volume 
$$V(C)=\frac{1}{n}V(P')\left(1-\frac{1}{n}\right)^{n-1}.$$
For $0\leq t\leq1$, let $P_{t}$ be the section of $P$ by the hyperplane $\{\mathbf{x}=(x_{1},\ldots,x_{n}) : x_{n}=t\}$.
%, and $V(P_{t})$ be the volume of the section $P_{t}$.
Then,
$$V(P)=\int_{0}^{1}V(P_{t})dt=\frac{1}{n}V(P_{\xi})=\frac{1}{n}V(P')=\left(1+\frac{1}{n-1}\right)^{n-1}V(C).$$

If $P$ is not a simplicial polytope, then $P$ is a pyramid. Without loss of generality, all but one vertex, say $\mathbf{v_{0}}$, are contained in a hyperplane $H_{0}=\aff\{\mathbf{v}_{1},\mathbf{v}_{2},\ldots,\mathbf{v}_{n},\mathbf{v}_{n+1}\}$. After applying a suitable affine transformation, we may assume that $H_{0}=\{\mathbf{x} \in \mathbb{R}^n : x_{n}=0\}$ and $\mathbf{v}_{0}=(0,\ldots,0,1)$. Writing $P''=\conv\{\mathbf{v}_{1},\mathbf{v}_{2},\ldots,\mathbf{v}_{n},\mathbf{v}_{n+1}\}$ for the base of the pyramid $P$, we construct the cylinder $$C=\conv\left\{\left((1-\frac{1}{n})P''+\frac{1}{n}\mathbf{v}_{0}\right)\bigcup\left((1-\frac{1}{n})P''\right)\right\} = \left(1-\frac{1}{n}\right)P'' + [\mathbf{0},\frac{1}{n}\mathbf{v}_0]$$
inscribed in $P$ with volume
$$V(C)=\frac{1}{n}V(P'')\left(1-\frac{1}{n}\right)^{n-1}.$$
Thus, we have $V(P)=\frac{1}{n}V(P'')=\left(1+\frac{1}{n-1}\right)^{n-1}V(C)$ as claimed.
\end{proof}

\begin{proof}[\textbf{Proof of Theorem~\ref{thm3}}]
The proof goes along the same lines as the proof of Theorem~\ref{thm1}, so that we only describe where the specifics of the polytope $P$ come into play.

By Lemma~\ref{lem5}, $P$ contains a cylinder $C$ whose base is an $(n-1)$-dimensional polytope with $n$ or $n+1$ vertices, and whose volume is given by $V(P)=\left(1+\frac{1}{n-1}\right)^{n-1}V(C)$.
By an inductive argument on the base of $C$, we can construct a $k$-cylinder inscribed in $P$.  
Combined with Lemma~\ref{lem1}, we obtain the desired bound on the density $\theta_L(P)$.
\end{proof}

\begin{rem}
\label{rem1}
Let $P$ be an $n$-dimensional polytope with $n+2$ vertices in $\mathbb{R}^{n}$.
If $n=2$, then $P$ is affinely equivalent to a two-dimensional anti-blocking body. However, this is not always true in higher dimensions.
For example, consider the triangle $T = \conv\{(1,0),(0,1),(-1,-1)\}$ in $\mathbb{R}^2$ and the $4$-dimensional polytope $Q = \conv\{T \times \{0\}^2 , \{0\}^2 \times T\}$.
By construction $Q$ has $6$ vertices and each vertex is contained in exactly~$5$ edges.
In order for $Q$ to be affinely equivalent to an anti-blocking polytope, however, there would need to exist a vertex of degree four.
\end{rem}

\section*{Acknowledgements}
The work of the third author is supported by the National Natural Science Foundation of China (NSFC12201307).

\end{document}